\documentclass{amsart}
\RequirePackage[utf8]{inputenc}

\usepackage{amsmath,amsthm,amsfonts,amssymb,amscd,amsbsy,dsfont,xfrac,hyperref}

\newcommand{\dd}{\mathrm{d}}
\newcommand{\dist}{\mathrm{dist}}
\newcommand{\hess}{\operatorname{Hess}}
\newcommand{\spa}{\operatorname{span}}
\newcommand{\R}{\mathds{R}}
\newcommand{\Z}{\mathds{Z}}
\newcommand{\C}{\mathds{C}}

\newcommand{\Ric}{\operatorname{Ric}}

\newcommand{\SO}{\mathsf{SO}}
\newcommand{\SU}{\mathsf{SU}}
\newcommand{\oc}{\sec^\perp}
\renewcommand{\v}{\mathcal{V}}
\newcommand{\h}{\mathcal{H}}
\newcommand{\id}{\operatorname{Id}}
\newcommand{\gr}{Gr_2TM}
\newcommand{\vol}{\,\mathrm{vol}}

\hypersetup{
    pdftoolbar=true,
    pdfmenubar=true,
    pdffitwindow=false,
    pdfstartview={FitH},
    pdftitle={Positive biorthogonal curvature on the product of 2-spheres},
    pdfauthor={Renato G. Bettiol},
    pdfsubject={Primary: 53C20, 53C21; Secondary: 53B21},
    pdfkeywords={}
    pdfnewwindow=true,
    colorlinks=true, 
    linkcolor=blue,
    citecolor=blue,
    urlcolor=blue,
}

\newtheorem{theorem}{Theorem}[]
\newtheorem*{thm}{Theorem}

\newtheorem{proposition}[theorem]{Proposition}
\newtheorem{corollary}[theorem]{Corollary}

\theoremstyle{definition}

\theoremstyle{remark}
\newtheorem{remark}[theorem]{Remark}

\title{Positive biorthogonal curvature on $S^2\times S^2$}
\author{Renato G. Bettiol}

\allowdisplaybreaks
\numberwithin{equation}{section}
\numberwithin{theorem}{section}

\address{University of Notre Dame \newline\indent Department of Mathematics \newline\indent 255 Hurley Building \newline\indent Notre Dame, IN, 46556-4618, USA}
\email{rbettiol@nd.edu}

\date{\today}
\thanks{The author is partially supported by the NSF grant DMS-0941615.}


\begin{document}
\begin{abstract}
We prove that $S^2\times S^2$ satisfies an intermediate condition between $\Ric>0$ and $\sec>0$. Namely, there exist metrics for which the average of the sectional curvatures of any two planes tangent at the same point, but separated by a minimum distance in the $2$-Grassmannian, is strictly positive; and this can be done with an arbitrarily small lower bound on the distance between the planes considered. Although they have positive Ricci curvature, these metrics do not have nonnegative sectional curvature. Such metrics also have positive \emph{biorthogonal curvature}, meaning that the average of sectional curvatures of any two orthogonal planes is positive. 
\end{abstract}

\maketitle

\section{Introduction}

Let $(M,g)$ be a $4$-dimensional Riemannian manifold. For each plane $\sigma\subset T_pM$ at a point $p\in M$, denote by $\sigma^\perp$ the orthogonal plane to $\sigma$, i.e., $\sigma\oplus\sigma^\perp=T_pM$ is a $g$-orthogonal direct sum. Define the \emph{biorthogonal (sectional) curvature} of $\sigma$ as the average of the sectional curvatures of $\sigma$ and $\sigma^\perp$, i.e.,
\begin{equation*}
\oc_g(\sigma):=\tfrac12\big(\sec_g(\sigma)+\sec_g(\sigma^\perp)\big).
\end{equation*}
The Hopf Conjecture, that asks if $S^2\times S^2$ admits a metric with $\sec>0$, remains one of the most intriguing open problems in Riemannian geometry. With the standard product metric $g_0$, at every point $p\in S^2\times S^2$ there exists $\sigma\subset T_pM$ with $\oc_{g_0}(\sigma)=0$. Namely, any \emph{mixed plane} $\sigma$ at $p$ (i.e., spanned by vectors of the form $(X,0)$ and $(0,Y)$) is such that $\sigma^\perp$  is also a mixed plane, hence $\sec_{g_0}(\sigma)=\sec_{g_0}(\sigma^\perp)=0$.
A natural question in this context is if the weaker condition $\oc>0$ can be satisfied in $S^2\times S^2$ \cite{ezio}.

The goal of this note is to give a positive answer, also covering a stronger curvature positivity condition, that can be defined in any dimension. Namely, choose a distance (inducing the standard topology) on the Grassmannian bundle $\gr$ of planes tangent to $M$, and for each $\theta>0$ and $\sigma\subset T_pM$, let
\begin{equation*}
\sec^\theta_g(\sigma):=\min_{\substack{\sigma'\subset T_pM \\ \dist(\sigma,\sigma')\geq\theta}}\tfrac12\big(\sec_g(\sigma)+\sec_g(\sigma')\big).
\end{equation*}

\begin{thm}
For every $\theta>0$, there exist Riemannian metrics $g^\theta$ on $S^2\times S^2$ with $\sec^{\theta}_{g^\theta}>0$, arbitrarily close to the standard product metric $g_0$ in the $C^k$-topology, $k\geq0$. In particular, $S^2\times S^2$ admits metrics of positive biorthogonal curvature.
\end{thm}

The condition $\sec^\theta_g>0$ means that at every point $p\in M$, the average of sectional curvatures of any two planes $\sigma_1,\sigma_2\subset T_pM$ that are at least $\theta>0$ apart from each other is positive. One can intuitively think of $\theta$ as a lower bound for the ``angle'' between the planes considered. Notice that if $\theta_1<\theta_2$, then $\sec^{\theta_1}_g>0$ clearly implies $\sec^{\theta_2}_g>0$. Furthermore, for every metric $g$ on $M$, there exists $\theta_g>0$ such that if $\sec^{\theta}_g>0$ for some $0<\theta\leq\theta_g$, then $\Ric_g>0$, see Proposition~\ref{prop:ricci}. In particular, for $4$-manifolds, if $\theta\leq\min_{p\in M,\,\sigma\subset T_pM}\dist(\sigma,\sigma^\perp)$, then $\sec^\theta_g>0$ implies $\oc_g>0$.

The construction of $g^\theta$ is so that these metrics converge to a limit metric $g^0$ as $\theta\to 0$ (possibly different from the product metric), in the $C^k$-topology, for any $k\geq0$. This convergence easily implies that, for $\theta>0$ sufficiently small, the metrics $g^\theta$ have positive Ricci curvature (Proposition~\ref{prop:ricci}) and positive biorthogonal curvature. In particular, the above theorem shows that a natural interpolating condition between $\Ric>0$ and $\sec>0$ is satisfied on $S^2\times S^2$.

We stress that $\oc>0$ alone does not imply $\Ric>0$, as illustrated by $S^1\times S^3$ with the standard product metric. This metric clearly has $\oc>0$, but since $S^1\times S^3$ has infinite fundamental group, it does not support any metrics with $\Ric>0$. Thus, in order to have a condition of this type that is stronger than $\Ric>0$, it is crucial that $\sec^\theta>0$ can be satisfied no matter how small $\theta>0$, and that the corresponding metrics converge.
In general, $\oc>0$ only implies positive scalar curvature, which poses some topological restrictions on $4$-manifolds (e.g., vanishing of all the Seiberg-Witten invariants), but these restrictions are by far not as strong as the ones implied by $\sec>0$ or $\Ric>0$. In particular, although $\oc>0$ is comparatively flexible, generic smooth $4$-manifolds do not support metrics with this property. Another indication of this relative flexibility of $\oc>0$ is that $\C P^2\#\overline{\C P}^2$ also admits metrics with this property (Proposition~\ref{prop:cp2}); while, similarly to $S^2\times S^2$, it remains an open question whether it admits a metric with $\sec>0$.

Our metrics $g^\theta$ with $\sec^\theta>0$ on $S^2\times S^2$ can be chosen invariant under the antipodal action of $\Z_2\oplus \Z_2$. Thus, for all $\theta>0$, the quotient $\R P^2\times \R P^2$ also admits metrics with $\sec^\theta>0$, arbitrarily close to the standard product metric. This illustrates a remarkable difference between $\sec^\theta>0$ (in particular, $\oc>0$) and $\sec>0$ since, by Synge's Theorem, $\R P^2\times \R P^2$ cannot have a metric with $\sec>0$. It is, however, somewhat expected that obstructions of Synge type do not detect these average curvature conditions, since even finiteness of the fundamental group goes unnoticed. 
We also remark that $(S^2\times S^2,g^\theta)$ has many points with planes of zero curvature (and even negative curvature), however any two such planes are always within distance $\theta$ from one another in the Grassmannian of planes tangent at that point.
In this way, $\theta$ corresponds to a measure of how big the regions formed by planes with nonpositive curvature can be in the Grassmannian. It would be interesting to know if metrics with $\sec^\theta>0$ on $S^2\times S^2$ can also be constructed while keeping $\sec\geq0$, as this could give a quantitative insight on the possibility of existence of quasipositively curved metrics.

The techniques used to construct all of the above metrics are (smooth) deformations.
Metric deformations to improve curvature have a long history, stemming from Berger and his students in the 1970's to the recent construction proposed by Petersen and Wilhelm \cite{pw,pwsphere} of a positively curved exotic sphere. Of particular importance in the present note are techniques developed by M\"uter~\cite{mueter} and Strake~\cite{strake,straket}, respectively regarding Cheeger deformations and deformations positive of first-order. The \emph{Cheeger deformation} is a method to attempt to increase curvature on nonnegatively curved manifolds with symmetries, by shrinking the metric in the direction of orbits of a large isometry group. This technique
was introduced by Cheeger~\cite{cheeger}, inspired by the construction of Berger metrics on odd-dimensional spheres, where the round metric is shrunk in the direction of the Hopf fibers. M\"uter~\cite{mueter} carried out a systematic study of Cheeger deformations in his PhD thesis under W. Meyer, establishing ground for a much better understanding of these deformed metrics. Strake~\cite{straket}, another PhD student of W. Meyer during the same period, studied metric deformations of nonnegatively curved metrics for which the first variation of the sectional curvature of any zero curvature plane is positive. These deformations are called \emph{positive of first-order}, and if the manifold is compact, they yield actual positively curved metrics. They also observed that, in this infinitesimal sense, Cheeger deformations are \emph{nonnegative} of first-order. 

Our deformation process from the product metric $g_0$ to a metric with $\sec^\theta>0$ has two steps, in which the above techniques are combined. The first is a Cheeger deformation, described in detail by M\"uter~\cite{mueter,zillermueter}. More precisely, we consider the cohomogeneity one diagonal $\SO(3)$-action on $S^2\times S^2$ and shrink $g_0$ in the direction of the orbits. This deformation gives a family of metrics $g_t$, $t>0$, with $\sec_{g_t}\geq0$ and much fewer planes of zero curvature than $g_0$. Namely, $(S^2\times S^2,g_t)$ has a circle's worth of zero curvature planes on points that lie on the diagonal or the anti-diagonal $\pm\Delta S^2=\{(p,\pm p):p\in S^2\}\subset S^2\times S^2$, and a unique zero curvature plane at any other point. This means that $\sec^\theta\geq0$, and equality holds only for some planes whose base point is in one of the submanifolds $\pm\Delta S^2$ (Proposition~\ref{prop:geomcheeger}).
Next, for fixed $t>0$, set $g:=g_t$. The second step is to employ a first-order local conformal deformation $g_s=g+s\,h$, where $h=\phi\, g$, and $\phi$ is supported in a tubular neighborhood of $\pm\Delta S^2$. Given the geometry of $(S^2\times S^2,g)$, we construct $\phi$ such that the first derivative with respect to $s$ of the average of two $g_s$-sectional curvatures is positive (Proposition~\ref{prop:final}). The function $\phi$ is proportional to the squared $g$-distance to $\pm\Delta S^2$, multiplied by a cutoff function. The strategy for such a construction is adapted from Strake~\cite{strake,straket}. Finally, a standard compactness argument (Proposition~\ref{prop:cpct}) implies that $\sec^\theta_{g_s}>0$ for all sufficiently small $s>0$, proving the desired result.

This paper is organized as follows. In Section~\ref{sec:cheeger}, we review basic aspects of Cheeger deformations, following M\"uter~\cite{mueter,zillermueter}. We describe the metric on $S^2\times S^2$ obtained by a Cheeger deformation with respect to the diagonal $\SO(3)$-action in terms of $\sec^\theta$. In Section~\ref{sec:firstorder}, we analyze the effects of a first-order deformation and construct the variation starting from the Cheeger deformed metric that proves the above Theorem. Some remarks on the geometry of the constructed metrics are given in Section~\ref{sec:remarks}. Finally, we briefly discuss $4$-manifolds with positive biorthogonal curvature (including the construction for $\C P^2\#\overline{\C P}^2$) in Section~\ref{sec:4mflds}.

\smallskip
\noindent
{\bf Acknowledgement.} It is a pleasure to thank Fernando Galaz-Garc\'ia, Karsten Grove, Paolo Piccione and Wolfgang Ziller for valuable comments and suggestions during the elaboration of this paper. We also express our sincere gratitude to the referee for the careful reading of the manuscript and constructive criticism.

\section{First step: Cheeger deformation}\label{sec:cheeger}

Although the techniques used in this section are mostly available elsewhere in the literature, see \cite{mueter,zillermueter,zillersurvey}, we briefly recall a few basic aspects as a service to the reader. For convenience, we use the same notation as the above references.

\subsection{Cheeger deformation}
Let $(M,g)$ be a Riemannian manifold and $\mathsf G$ a compact Lie group that acts on $M$ by isometries. The \emph{Cheeger deformation} of $g$ is a $1$-parameter family $g_t$, $t\geq0$, of $\mathsf G$-invariant metrics on $M$, defined as follows. Let $Q$ be a bi-invariant metric on $\mathsf G$, and endow $M\times\mathsf G$ with the product metric $g+\tfrac1t Q$. Consider the submersion
\begin{equation}
\rho\colon M\times\mathsf G\to M, \quad \rho(p,h)=h^{-1}p,
\end{equation}
and define $g_t$ as the metric on $M$ that turns $\rho$ into a Riemannian submersion. The family of metrics $g_t$ extends smoothly across $t=0$, with $g_0=g$, thus providing a deformation of such metric. Since $\sec_Q\geq0$, it follows immediately from the Gray-O'Neill formula that, if $\sec_{g_0}\geq0$, then also $\sec_{g_t}\geq0$, $t\geq0$. As we will see, many planes with zero curvature with $g_0$ usually gain positive curvature with $g_t$. 

For each $p\in M$, denote by $\mathsf G_p$ the isotropy group at $p$ and by $\mathfrak g_p$ its Lie algebra. Fix the $Q$-orthogonal splitting $\mathfrak g=\mathfrak g_p\oplus\mathfrak m_p$, and identify $\mathfrak m_p$ with the tangent space $T_p\mathsf G(p)$ to the $\mathsf G$-orbit through $p$ via action fields. More precisely, we identify $X\in\mathfrak m_p$ with $X^*_p=\frac{\dd}{\dd s}\exp(sX)p\big|_{s=0}\in T_p\mathsf G(p)$. This determines a $g_t$-orthogonal splitting
$T_pM=\v_p\oplus\h_p$ in \emph{vertical space} $\v_p:=T_p\mathsf G(p)=\{X^*_p:X\in\mathfrak m_p\}$ and \emph{horizontal space} $\h_p:=\{v\in T_pM : g_t(v,\v_p)=0\}$. Notice that the dimensions of $\v_p$ and $\h_p$ may vary with $p\in M$, hence these are not distributions.

Let $P_t\colon\mathfrak m_p\to\mathfrak m_p$ be the $Q$-symmetric automorphism that relates the metrics $Q$ and $g_t$, i.e., such that
\begin{equation}\label{eq:p}
Q(P_t(X),Y)=g_t(X^*_p,Y^*_p), \quad X,Y\in\mathfrak m_p.
\end{equation}
It is an easy computation that $P_t$ is determined by $P_0$ in the following way:
\begin{equation}\label{eq:pt}
P_t=(P_0^{-1}+t\id)^{-1}=P_0\,(\id+tP_0)^{-1}, \quad t\geq0,
\end{equation}
see \cite[Prop 1.1]{zillermueter}.
Thus, if we let $C_t\colon T_pM\to T_pM$ be the $g$-symmetric automorphism that relates $g$ and $g_t$, i.e., such that 
\begin{equation}\label{eq:ctdef}
g(C_t(X),Y)=g_t(X,Y), \quad X,Y\in T_pM,
\end{equation}
we then get
\begin{equation}\label{eq:ct}
C_t(X)=P_0^{-1}P_t(X^\v)+X^\h, \quad X\in T_pM,
\end{equation}
where $X^\v$ and $X^\h$ are the vertical and horizontal components of $X$ respectively. This reveals how the geometry of $g_t$ changes with $t$, since if $P_0$ has eigenvalues $\lambda_i$, then $C_t$ has eigenvalues $\frac{1}{1+t\lambda_i}$ corresponding to the vertical directions and eigenvalues $1$ in the horizontal directions. In other words, as $t$ grows, the metric $g_t$ shrinks in the direction of the orbits and remains the same in the orthogonal directions.

\subsection{Curvature evolution}
Let us now analyze how the curvature changes under this deformation. Henceforth, we assume the initial metric $g_0$ has $\sec_{g_0}\geq0$. As explained above, this implies $\sec_{g_t}\geq0$ for all $t\geq0$. Given $X\in T_pM$, denote by $X_{\mathfrak m}$ the unique vector in $\mathfrak m_p$ such that $(X_{\mathfrak m})^*_p=X_p^\v$. Also, given a plane $\sigma=\spa\{X,Y\}$, we write
\begin{equation*}
C_t^{-1}(\sigma):=\spa\{C_t^{-1}X,C_t^{-1}Y\}.
\end{equation*}
As explained by Ziller \cite{zillermueter}, the crucial observation of M\"uter is that, to analyze the evolution of $\sec_{g_t}$, it is much more convenient to study $\sec_{g_t}(C_t^{-1}(\sigma))$ rather than $\sec_{g_t}(\sigma)$. In more recent literature, the $1$-parameter family of bundle automorphisms induced by $C_t^{-1}$ in the Grassmannian bundle $\gr$ of planes on $M$ is being called \emph{Cheeger reparametrization}, see \cite{pw,pwsphere}. The following result of M\"uter~\cite[Satz 3.10]{mueter} (see also \cite[Cor 1.4]{zillermueter}) summarizes how the curvature of $g_t$ evolves.

\begin{proposition}\label{prop:mueter}
Let $g_t$ be the Cheeger deformation of $g_0$, $\sec_{g_0}\geq0$. Given $\sigma=\spa\{X,Y\}\subset T_pM$, consider the unnormalized $g_t$-sectional curvature of $C_t^{-1}(\sigma)$:
\begin{equation*}
k_c(t):=\big\|C_t^{-1}X\wedge C_t^{-1}Y\big\|^2_{g_t}\sec_{g_t}(C_t^{-1}(\sigma))=g_t\big(R_t(C_t^{-1}X,C_t^{-1}Y)C_t^{-1}Y,C_t^{-1}X\big),
\end{equation*}
where $R_t$ is the curvature tensor of $g_t$. If $\sec_{g_0}(\sigma)=0$, then:
\begin{itemize}
\item[(i)] $k'_c(0)\geq0$;
\item[(ii)] If $k'_c(0)=0$ and $[P_0 X_{\mathfrak m},P_0Y_{\mathfrak m}]\neq0$, then $k_c''(0)=0$, $k_c'''(0)>0$ and $k_c(t)>0$ for all $t>0$;
\item[(iii)] If $k'_c(0)=0$ and $[P_0 X_{\mathfrak m},P_0 Y_{\mathfrak m}]=0$, then $k_c(t)=0$ for all $t>0$.
\end{itemize}
In particular, if $[P_0 X_{\mathfrak m}, P_0 Y_{\mathfrak m}]\neq0$, i.e., the plane $\spa\{P_0X_{\mathfrak m},P_0Y_{\mathfrak m}\}$ has positive curvature in $(\mathsf G,Q)$, and $\sec_{g_0}(\sigma)=0$, then $\sec_{g_t}(C_t^{-1}(\sigma))>0$ for all $t>0$.
\end{proposition}

Observe that, from (iii), if $\sigma$ is tangent to a totally geodesic flat torus in $M$ that contains a horizontal direction, then $\sec_{g_t}(C_t^{-1}(\sigma))=0$, $t\geq0$, i.e., $\sigma$ remains flat. On the other hand, we also get the following important positive result:

\begin{corollary}\label{cor:so3}
Assume $\mathsf G=\SO(3)$ or $\mathsf G=\SU(2)$, so that $\sec_Q>0$. If $\sec_{g_0}(\sigma)=0$ and the image of the projection of $\sigma\subset\v_p\oplus\h_p$ onto $\v_p$ is $2$-dimensional, then $\sec_{g_t}(C_t^{-1}(\sigma))>0$ for all $t>0$. In other words, up to the Cheeger reparametrization, zero curvature planes with nondegenerate vertical projection have positive curvature with $g_t$, for all $t>0$.
\end{corollary}

\subsection{The case of $S^2\times S^2$}
Consider $S^2\times S^2$ endowed with the standard product metric $g_0$ and the diagonal $\SO(3)$-action:
\begin{equation*}
A\cdot(p_1,p_2)=(A\,p_1,A\,p_2), \quad p=(p_1,p_2)\in S^2\times S^2\subset\R^3\oplus\R^3, A\in\SO(3).
\end{equation*}
This is a cohomogeneity one isometric action with orbit space a closed interval, so there are codimension one principal orbits (corresponding to interior points of the interval) and two singular orbits (corresponding to the endpoints), see \cite{zillersurvey}. These singular orbits are the \emph{diagonal} and \emph{anti-diagonal} submanifolds:
\begin{equation*}
\pm\Delta S^2:=\{(p_1,\pm p_1):p_1\in S^2\}\subset S^2\times S^2.
\end{equation*}
The principal isotropy $\mathsf G_p$, $p\not\in\pm\Delta S^2$, is trivial, since it consists of orientation-preserving isometries of $\R^3$ that fix two linearly independent directions. The singular isotropies are formed by orientation-preserving isometries of $\R^3$ that fix one direction, hence are isomorphic to $\SO(2)$. Thus, the group diagram of this action is $\{1\}\subset\{\SO(2),\SO(2)\}\subset\SO(3)$.

Following M\"uter~\cite{mueter}, we identify the Lie algebra of $\SO(3)$ with $\R^3$ by:
\begin{equation*}
\mathfrak{so}(3)\ni Z=\left(\begin{array}{ccc}
0 & -z_3 & z_2 \\
z_3 & 0 & -z_1 \\
-z_2 & z_1 & 0
\end{array}\right)\longleftrightarrow\left(\begin{array}{c}
z_1 \\ z_2 \\ z_3\end{array}\right)=z\in\R^3.
\end{equation*}
Considering $(\mathfrak{so}(3),Q)$ endowed with the standard bi-invariant metric, the above is an isometric identification with Euclidean space $(\R^3,\langle\cdot,\cdot\rangle)$. In this way, since the Lie exponential in $\SO(3)$ is given by matrix exponentiation, the action field induced by $Z\in\mathfrak{so}(3)$ is:
\begin{equation}\label{eq:actionfield}
Z^*_p=(Z^*_{p_1},Z^*_{p_2}) =(Z\,p_1,Z\,p_2)=(z\wedge p_1,z\wedge p_2)\in T_p(S^2\times S^2).
\end{equation}
So, if $x,y\in\R^3$ are such that $\langle x,p_1\rangle=\langle y,p_2\rangle=0$, then for all $z\in\R^3$,
\begin{equation}\label{eq:coiso}
g_0\big((X^*_{p_1},Y^*_{p_2}),Z^*_p\big)=\langle x\wedge p_1,z\wedge p_1\rangle +\langle y\wedge p_2,z\wedge p_2\rangle =\langle x+y,z\rangle.
\end{equation}
Thus, the vector $(X^*_{p_1},-X^*_{p_2})\in T_p(S^2\times S^2)$ is horizontal, whenever $x\in\{p_1,p_2\}^\perp:=\{x\in\R^3:\langle x,p_1\rangle=\langle x,p_2\rangle=0\}$. By dimensional reasons, it then follows that the horizontal space for the $\SO(3)$-action on $S^2\times S^2$ at $p=(p_1,p_2)$ is
\begin{equation*}
\h_p=\big\{(X^*_{p_1},-X^*_{p_2})\in T_p(S^2\times S^2):x\in\{p_1,p_2\}^\perp\big\}.
\end{equation*}
Recall that the vertical space is $\v_p=\{(X^*_{p_1}, X^*_{p_2}):X\in\mathfrak m_p\}$, see \eqref{eq:actionfield}. For general $x,y\in\R^3$, analogously to \eqref{eq:coiso}, we have:
\begin{multline*}
g_0\big(X^*_{p},Y^*_{p}\big)=g_0\big((X^*_{p_1},X^*_{p_2}),(Y^*_{p_1},Y^*_{p_2})\big)=\langle x\wedge p_1,y\wedge p_1\rangle +\langle x\wedge p_2,y\wedge p_2\rangle =\\ 
\langle p_1\wedge(x\wedge p_1),y\rangle+\langle p_2\wedge(x\wedge p_2),y\rangle =\big\langle (2x-\langle x,p_1\rangle p_1-\langle x,p_2\rangle p_2),y\big\rangle.
\end{multline*}
From \eqref{eq:p}, the above is equal to $\langle P_0\, X,Y\rangle$, so we get an explicit formula for $P_0\colon\mathfrak m_p\to\mathfrak m_p$ in our example:
\begin{equation*}
P_0\, X=2X-\langle X,p_1\rangle p_1-\langle X,p_2\rangle p_2.
\end{equation*}
In particular, it follows that the subspace $\{p_1,p_2\}^\perp\subset \mathfrak m_p$ is invariant under $P_0$ and hence under $P_t$ and $C_t$, see \eqref{eq:pt} and \eqref{eq:ct}.

Let $\pi\colon\{p_1,p_2\}^\perp\to\{p_1,p_2\}^\perp/\sim$ be the projection onto the corresponding real projective space. For each $x\in\{p_1,p_2\}^\perp$, consider the vertizontal\footnote{i.e., this plane is spanned by one vertical and one horizontal vector.} plane
\begin{equation}\label{eq:sigmapi}
\sigma_{\pi(x)}:=\spa\{(X^*_{p_1},0),(0,X^*_{p_2})\}=\spa\{(X^*_{p_1},X^*_{p_2}),(X^*_{p_1},-X^*_{p_2})\}.
\end{equation}
This is the unique mixed plane at $p$ that contains the horizontal vector $(X^*_{p_1},-X^*_{p_2})$. Thus, from Corollary~\ref{cor:so3}, $\{\sigma_{\pi(x)}\subset T_p(S^2\times S^2):x\in\{p_1,p_2\}^\perp\}$ are the only $g_0$-flat planes at $p$ that remain $g_t$-flat for $t>0$, up to the Cheeger reparametrization. As a matter of fact, by the above, the Cheeger reparametrization fixes such planes, i.e., 
\begin{equation}\label{eq:ctfix}
C_t^{-1}(\sigma_{\pi(x)})=\sigma_{\pi(x)}.
\end{equation} 

In conclusion, for any $t>0$, $\sec_{g_t}\geq0$ and $\sec_{g_t}(\sigma)=0$ if and only if $\sigma=\sigma_{\pi(x)}\subset T_p(S^2\times S^2)$ for some $x\in\{p_1,p_2\}^\perp$. In particular, $g_t$-flat planes at $p$ are parametrized by $\pi(\{p_1,p_2\}^\perp)$. Thus, in $(S^2\times S^2,g_t)$, $t>0$, there is a circle's worth of zero curvature planes at each $p\in\pm\Delta S^2$, a unique zero curvature plane at each $p\not\in\pm\Delta S^2$, and all other planes have positive curvature. In terms of $\sec^\theta$, we thus have the following.

\begin{proposition}\label{prop:geomcheeger}
Let $g_t$ be the Cheeger deformation of $(S^2\times S^2,g_0)$ with respect to the diagonal $\SO(3)$-action. 
Then, for any $\theta>0$ and $t>0$, $\sec^\theta_{g_t}\geq0$ and equality only holds for planes at some $p=(p_1,\pm p_1)\in\pm\Delta S^2$. If $\sigma\subset T_{p}(S^2\times S^2)$ has $\sec^\theta_{g_t}(\sigma)=0$, then $\sigma=\sigma_{\pi(x)}$ for some $x\in\{p_1\}^\perp$; in particular, $\sigma$ is not tangent to the submanifold $\pm\Delta S^2$.
\end{proposition}

\begin{proof}
The above statements follow immediately from M\"uter~\cite[Satz 4.26]{mueter} (see also Ziller~\cite[p. 5]{zillermueter} and Kerin~\cite[Rem 4.3]{kerin}), as well as from the above discussion.
\end{proof}

\begin{remark}
For $n\geq3$, although there exists an analogous cohomogeneity one $\SO(n+1)$-action on $S^n\times S^n$, the corresponding Cheeger deformation fails to produce so many positively curved planes. This is due to the fact that $\SO(n+1)$, $n\geq 3$, is not positively curved, cf. Corollary~\ref{cor:so3}. As a result, this step in the construction of our metrics with $\sec^\theta>0$ only works on $S^n\times S^n$ if $n=2$.
\end{remark}

\section{Second step: first-order local conformal deformation}\label{sec:firstorder}

As seen above, the Cheeger deformed metrics $g_t$, $t>0$, have $\sec^\theta_{g_t}\geq0$ and equality holds only for certain planes (of the form \eqref{eq:sigmapi}) at $\pm\Delta S^2$. In order to get these planes to also have $\sec^\theta>0$, we now carry out a (local) first-order conformal deformation, inspired by results of Strake~\cite{strake}. More precisely, choose $g$ to be a Cheeger deformed metric $g_t$ for any $t>0$ and consider the new $1$-parameter family
\begin{equation}\label{eq:gs}
g_s:=g+s\,h, \quad s\in\,]-\varepsilon,\varepsilon[,
\end{equation}
where $h$ is some symmetric $(0,2)$-tensor to be defined, and $\varepsilon>0$ is small enough so that $g_s$ is still a Riemannian metric. Given the above geometry of the Cheeger deformed metric $g$, we will choose $h$ such that 
\begin{multline}\label{eq:goal}
\frac{\dd}{\dd s}\big(\sec_{g_s}(\sigma_1)+\sec_{g_s}(\sigma_2)\big)\Big|_{s=0}>0, \quad \mbox{ for all } \sigma_1,\sigma_2\subset T_p(S^2\times S^2) \mbox{ with }\\ p\in\pm\Delta S^2 \mbox{ and }\sec_g(\sigma_1)=\sec_g(\sigma_2)=0.
\end{multline}
The crucial observation that makes this possible is that these planes are never tangent to $\pm\Delta S^2$. Our choice will be such that $h$ is supported only near $\pm\Delta S^2$ and is pointwise proportional to $g$, justifying the terminology.
We start by recalling the first variation of $\sec_{g_s}(\sigma)$, see Strake~\cite[Sec 3.a]{strake}.

\begin{proposition}\label{prop:secondstep}
Let $(M,g)$ be a Riemannian manifold with $\sec_g\geq0$ and $X,Y\in T_pM$ be $g$-orthonormal vectors that span a $g$-flat plane $\sigma\subset T_pM$. Consider a first-order variation $g_s=g+s\,h$. Then the first variation of $\sec_{g_s}(\sigma)$ is given by
\begin{equation*}
\frac{\dd}{\dd s}\sec_{g_s}(\sigma)\Big|_{s=0}=\nabla_X\nabla_Y h(X,Y)-\tfrac12\nabla_X\nabla_X h(Y,Y)-\tfrac12\nabla_Y\nabla_Y h(X,X).
\end{equation*}
In particular, if $h=\phi\,g$, then
\begin{equation}\label{eq:dds}
\frac{\dd}{\dd s}\sec_{g_s}(\sigma)\Big|_{s=0} =-\tfrac12 \hess\phi\,(X,X)-\tfrac12\hess\phi\,(Y,Y).
\end{equation}
\end{proposition}

Now, observe that if $N\subset M$ is an embedded submanifold, the squared distance function $\psi(p)=\dist_g(p,N)^2$ is smooth in a sufficiently small tubular neighborhood of $N$. The gradient of $\psi$ at $p$ vanishes if $p\in N$, and points in the outward radial direction if $p\not\in N$. The Hessian of $\psi$ at $p\in N$ is given by:
\begin{equation}\label{eq:hess}
\hess\psi\,(X,X)=2g(X_\perp,X_\perp)=2\|X_\perp\|^2_g, \quad X\in T_pM,
\end{equation}
where $X=X_\top+X_\perp\in T_pN\oplus (T_pN)^\perp$ is the $g$-orthogonal decomposition in tangent and normal parts to $N$.

\begin{proposition}\label{prop:final}
Consider the metrics $g_s$ on $S^2\times S^2$, given by \eqref{eq:gs}. There exists a smooth function $\phi\colon S^2\times S^2\to\R$, supported in a neighborhood of $\pm\Delta S^2$, such that if $h=\phi\,g$, then \eqref{eq:goal} holds.
\end{proposition}

\begin{proof}
From Proposition~\ref{prop:geomcheeger}, the only planes $\sigma$ with $\sec^\theta_g(\sigma)=0$ are of the form $\sigma_{\pi(x)}\subset T_{(p_1,\pm p_1)}(S^2\times S^2)$ for some $x\in\{p_1\}^\perp$. These planes are vertizontal, i.e., they contain a direction normal to $\pm\Delta S^2$.

Let us analyze the plane $\sigma_{\pi(x)}$  at $(p_1,p_1)\in\Delta S^2$, the case of $(p_1,-p_1)\in-\Delta S^2$ being totally analogous. As mentioned above, the function $\psi_+(p)=\dist_g(p,\Delta S^2)^2$ is smooth in a tubular neighborhood $D(\Delta S^2)$. Let $\chi_+\colon S^2\times S^2\to\R$ be a smooth cutoff function that vanishes outside $D(\Delta S^2)$ and is equal to $1$ in a smaller tubular neighborhood of $\Delta S^2$. 
According to \eqref{eq:sigmapi}, \eqref{eq:dds} and \eqref{eq:hess}, if we set $h_+=-(\chi_+\psi_+)\,g$ and $g^+_s=g+s\,h_+$, then
\begin{equation*}
\frac{\dd}{\dd s}\sec_{g^+_s}(\sigma_{\pi(x)})\Big|_{s=0}=\big\|(X^*_{p_1},-X^*_{p_1})\big\|^2_g>0.
\end{equation*}
Defining $\psi_-$ and $\chi_-$ analogously, we get $h_-=-(\chi_-\psi_-)\,g$ and $g^-_s=g+s\,h_-$ with the same property as above for planes at $(p_1,-p_1)\in -\Delta S^2$ with zero $g$-curvature. Thus, the function $\phi:=-(\chi_+\psi_+)-(\chi_-\psi_-)$ has the desired properties. More precisely, $h=\phi\,g=h_+ +h_-$ is such that $g_s=g+s\,h$ coincides with $g_s^\pm$ near $\pm\Delta S^2$. Hence $\frac{\dd}{\dd s}\sec_{g_s}(\sigma)\big|_{s=0}>0$ for all planes $\sigma$ at $\pm\Delta S^2$ such that $\sec_g(\sigma)=0$; in particular, \eqref{eq:goal} holds.
\end{proof}

In order to conclude the proof of the Theorem in the Introduction, we quote the following elementary fact.

\begin{proposition}\label{prop:cpct}
Let $f\colon \,]-\varepsilon,\varepsilon[\,\times K\to\R$ be a smooth function, where $K$ is a compact manifold (possibly with boundary). Assume that $f(0,x)\geq 0$ for all $x\in K$ and $\frac{\partial f}{\partial s}(0,x)>0$ if $f(0,x)=0$. Then, there exists $s_*>0$ such that $f(s,x)>0$ for all $x\in K$ and $0<s<s_*$.
\end{proposition}

\begin{proof}
Follows by a routine compactness argument, using the Taylor polynomial of $f(s,x)$ at $s=0$.
\end{proof}

Consider the following compact subset of the fibered product of two copies of the Grassmannian bundle $\gr$, where $M=S^2\times S^2$:
\begin{equation*}
K_\theta:=\big\{(p,\sigma_1,\sigma_2)\in M\times\gr\times\gr : \sigma_1,\sigma_2\subset T_pM, \,\dist(\sigma_1,\sigma_2)\geq\theta\big\}.
\end{equation*}
Due to \eqref{eq:goal} and $\theta>0$, the function
\begin{equation*}
f\colon \,]-\varepsilon,\varepsilon[\,\times K_\theta\longrightarrow\R, \quad f(s,\sigma_1,\sigma_2)=\tfrac12\big(\sec_{g_s}(\sigma_1)+\sec_{g_s}(\sigma_2)\big),
\end{equation*}
satisfies the hypotheses of Proposition~\ref{prop:cpct}, hence there exists $s_\theta>0$ such that $f(s,\sigma_1,\sigma_2)>0$ for all $(\sigma_1,\sigma_2)\in K_\theta$ and $0<s<s_\theta$. This is equivalent to $\sec^\theta_{g_s}>0$, for $0<s<s_\theta$; so we can finally choose $g^\theta:=g_s$, for any $0<s<s_\theta$. Moreover, it follows immediately from the above results (Propositions~\ref{prop:geomcheeger}, \ref{prop:secondstep} and \ref{prop:cpct}) that  $g^\theta$ can be constructed arbitrarily close to the standard product metric $g_0$, in any $C^k$-topology, $k\geq0$, by choosing $t>0$ (the duration of the Cheeger deformation) and $s>0$ sufficiently small. This concludes the proof of the Theorem in the Introduction.

\section{Remarks on the construction}\label{sec:remarks}

\subsection{First-order deformations and the Hopf conjecture}
The above first-order deformation $g_s$ works to get $\sec^\theta>0$ on all of $M=S^2\times S^2$ because the only points $p\in M$ that have planes $\sigma_1,\sigma_2\subset T_pM$ with $f(0,\sigma_1,\sigma_2)=0$ are contained in the submanifolds $\pm\Delta S^2$, which admit a relatively compact neighborhood and where $\frac{\partial f}{\partial s}(0,\sigma_1,\sigma_2)>0$. The same cannot be done for the sectional curvature because at \emph{every point} there is a plane $\sigma$ with $\sec_g(\sigma)=0$. The only type of first-order deformation that would give $\sec_{g_s}>0$ would be one with $\frac{\dd}{\dd s}\sec_{g_s}(\sigma)\big|_{s=0}>0$ for all $\sigma$ with $\sec(\sigma)=0$. It was proved by Strake~\cite[Prop. 4.3]{strake} that such a deformation does not exist on $(S^2\times S^2,g)$, due to the presence of totally geodesic flat tori.

\subsection{Other compact subsets}\label{subsec:othercpcts}
Notice that condition \eqref{eq:goal} does not contain any information on the compact subset $K_\theta$, which is the domain considered for the function $f(s,\sigma_1,\sigma_2)=\tfrac12\big(\sec_{g_s}(\sigma_1)+\sec_{g_s}(\sigma_2)\big)$. This means that the same argument above could be applied to obtain positivity of the average of sectional curvatures of planes that satisfy some other conditions codified in the form of a compact subset 
\begin{equation*}
K\subset \big\{(p,\sigma_1,\sigma_2)\in M\times\gr\times\gr : \sigma_1,\sigma_2\subset T_pM\big\},
\end{equation*}
where $M=S^2\times S^2$. Replacing the domain of $f$ by $K$, provided that $K$ does not intersect the diagonal (i.e., the subset $\Delta=\{(p,\sigma,\sigma):\sigma\subset T_pM\}$), we get from Proposition~\ref{prop:cpct} that $f(s,\sigma_1,\sigma_2)>0$ for $s>0$ small enough and all $(\sigma_1,\sigma_2)\in K$. We must require that $K$ be away from the diagonal, otherwise $f(0,\sigma_1,\sigma_2)$ would also have zeros on points outside the singular orbits $\pm\Delta S^2$, and there is no first-order variation that accounts for $\frac{\partial f}{\partial s}(0,\sigma_1,\sigma_2)>0$ at all such points. Notice also that for every $K$ with the required properties above, there exists $\theta>0$ such that $K\subset K_\theta$, so all other possibilities are accounted for by using the domains $K_\theta$.

\subsection{Ricci curvature}\label{subsec:ricci}
Since we know that $g^\theta$ can be constructed arbitrarily $C^k$-close (for any $k\geq0$) to the product metric $g_0$, it automatically follows that such metrics can be chosen with positive Ricci curvature. Nevertheless, existence of metrics with $\Ric>0$ can be directly deduced from the existence of metrics with $\sec^\theta>0$ for arbitrarily small $\theta>0$, that converge to a limit metric as $\theta\to0$, in the $C^k$-topology, $k\geq0$, as we shall now prove. This abstract property is hence stronger\footnote{Notice that the same is not true for $\oc>0$ on $4$-manifolds, as illustrated by $S^1\times S^3$ with the standard product metric. The crucial point is that $\sec^\theta>0$ has to be satisfied no matter how small $\theta>0$, and the metrics that do so do not diverge or degenerate.} than $\Ric>0$ for compact manifolds (and, of course, weaker than $\sec>0$), regardless of the dimension of $M$. In this way, the Theorem in the Introduction shows that a natural interpolating condition between $\Ric>0$ and $\sec>0$ is satisfied on $S^2\times S^2$.

\begin{proposition}\label{prop:ricci}
Let $M$ be a compact $n$-dimensional manifold such that for every $\theta>0$ there exists a metric $g^\theta$ with $\sec^\theta_{g^\theta}>0$. Assume that there exists a metric $g^0$ on $M$ such that $g^\theta\to g^0$ in the $C^0$-topology, as $\theta\to 0$. Then $\Ric_{g^\theta}>0$ for $\theta>0$ sufficiently small; in particular, if $g^\theta\to g^0$ also in the $C^2$-topology, then $\Ric_{g^0}\geq0$.\footnote{Moreover, in this case, one can also easily prove that $\sec_{g^0}\geq0$.}
\end{proposition}

\begin{proof}
For any metric $g$ on $M$, define:
\begin{equation}
\theta_g:=\min_{p\in M}\left(\min_{\substack{v\in T_pM, \\ g(v,v)=1}}\left(\min_{\substack{w_1,w_2\in T_pM, \\ g(v,w_j)=0, \\ g(w_i,w_j)=\delta_{ij}}} \dist\Big(\spa\{v,w_1\},\spa\{v,w_2\}\Big)\right)\right).
\end{equation}
The above defines a positive number, that depends continuously on the metric $g$, such that if $\sec^{\theta}_g>0$ for some $0<\theta\leq\theta_g$, then $\Ric_g>0$. In fact, $\Ric_g(v)>0$ for any direction $v$, since this is a sum of $(n-1)$ sectional curvatures whose pairwise average is positive, because $\sec^\theta_g>0$.

Given the continuous family $\mathcal G:=\{g^\theta:\theta\in [0,1]\}$, let $\theta_*:=\min \{\theta_g: g\in\mathcal G\}$.
 It then follows that $\theta_*>0$ and hence for any $0<\theta\leq\theta_*$, we have $\Ric_{g^\theta}>0$.
\end{proof}

\begin{remark}
An immediate consequence of the above is that, although $S^1\times S^3$ has a metric with $\oc>0$, it cannot satisfy $\sec^\theta>0$ for all $\theta>0$ with metrics that do not diverge (otherwise, it would have a metric with $\Ric>0$).\end{remark}

\subsection{Negative sectional curvatures}\label{subsec:negatives}
Although the first step in our deformation preserves $\sec\geq0$ from the product metric, the second step does not. In fact, for all $\theta>0$ there are planes $\sigma$ in $(S^2\times S^2,g^\theta)$ with $\sec_{g^\theta}(\sigma)<0$. This follows from an obstruction to positive first-order deformations observed by Strake~\cite[Sec. 4]{strake}. Namely, all zero planes in the Cheeger deformed metric $g=g_t$ from Section~\ref{sec:cheeger} are tangent to a totally geodesic flat torus, see M\"uter~\cite[Satz 4.26]{mueter}. Pick one such torus $i\colon T^2\hookrightarrow (S^2\times S^2,g)$, that intersects $\pm\Delta S^2$. The first-order deformation $g_s=g+s\,h$ on $S^2\times S^2$ induces a first-order deformation $i^*g_s$ on $T^2$. As observed by Strake~\cite[Lemma 4.1]{strake}, since $i(T^2)$ is totally geodesic, the first variation for the sectional curvature on $T^2$ coincides with the ambient variation:
\begin{equation}\label{eq:ddsinduced}
\frac{\dd}{\dd s}\sec_{i^*g_s}(\sigma)\Big|_{s=0}=\frac{\dd}{\dd s}\sec_{g_s}(\dd i(\sigma))\Big|_{s=0}.
\end{equation}
In fact, this follows directly by differentiating the Gauss equation of $i(T^2)\subset (S^2\times S^2,g_s)$ at $s=0$. Let $i(p)$ be a point where $i(T^2)$ intersects $\pm\Delta S^2$. Then if $\sigma=T_pT^2$, $\dd i(\sigma)$ is such that $\sec^\theta_g(\dd i(\sigma))=0$, so the construction in Section~\ref{sec:firstorder} is such that \eqref{eq:ddsinduced} is positive. By the Gauss-Bonnet Theorem, $A(s)=\int_{T^2}\sec_{i^*g_s}\vol_{i^*g_s}=2\pi\chi(T^2)$ vanishes identically, so that
\begin{equation}
0=A'(0)=\int_{T^2} \left(\frac{\dd}{\dd s}\sec_{i^*g_s}\Big|_{s=0}\right) \vol_{i^*g}.
\end{equation}
Since the above integrand is positive at $i(p)\in\pm\Delta S^2$, it must also be negative somewhere. Together with \eqref{eq:ddsinduced} and the fact that $i(T^2)\subset(S^2\times S^2,g)$ is totally geodesic and flat, this means that $g_s$, $s>0$, must have some negative sectional curvature.

\subsection{Limiting case}
Since $\theta>0$ can be chosen arbitrarily small for our construction, a natural question is what happens to $g^\theta$ as $\theta\to 0$. By the above observations, the metric $g_s$ in \eqref{eq:gs} has some negative sectional curvature as soon as $s>0$. This implies that as $\theta\to 0$, the interval $0<s<s_\theta$ for which $g_s$ has $\sec^\theta>0$ shrinks until it disappears when $\theta=0$, since $s_\theta$ must also go to zero. In fact, if there was a uniform lower bound $0<s_*\leq s_\theta$ for all $\theta>0$, then the metrics $g_s$, $0<s<s_*$, would be such that the average sectional curvatures of any two distinct planes at the same point is positive, which in particular implies $\sec_{g_s}\geq0$, $0<s<s_*$, contradicting Subsection~\ref{subsec:negatives}. This is also reflected by the fact that the domain $K$ must be chosen compact in order for Proposition~\ref{prop:cpct} to hold, hence one cannot simply take $K$ to be the complement of the diagonal, see also Subsection~\ref{subsec:othercpcts}.

\subsection{Finite quotient}\label{subsec:rp2}
Our construction of metrics $g^\theta$ with $\sec^\theta>0$ on $S^2\times S^2$ can be made invariant under the antipodal action of $\Z_2\oplus\Z_2$, so that they induce metrics with $\sec^\theta>0$ on $\R P^2 \times \R P^2$. In particular, $\R P^2 \times \R P^2$ admits metrics with $\oc>0$.  Since such metrics come from a local isometric covering $(S^2\times S^2,g^\theta)\to\R P^2\times \R P^2$, they also do not have $\sec\geq0$ due to the above observations.

The first step in the construction gives rise to metrics invariant under $\Z_2\oplus\Z_2$, since it is a Cheeger deformation with respect to the $\SO(3)$-action, which commutes with the $\Z_2\oplus\Z_2$-action. As a side note, it was observed by M\"uter \cite[Satz 4.27]{mueter} that the induced metric on $\R P^2\times \R P^2$ at this stage is such that all its zero curvature planes are tangent to totally geodesic flat tori. The second and final step of our construction can also be made so that the resulting metrics are $\Z_2\oplus\Z_2$-invariant. Namely, this property is equivalent to the function $\phi\colon S^2\times S^2\to\R$ in Proposition~\ref{prop:final} being $\Z_2\oplus\Z_2$-invariant, which can be achieved by defining the cutoff functions $\chi_\pm$ in a symmetric way. 

\subsection{Biorthogonal pinching and isotropic curvature}
The biorthogonal curvature of a manifold $(M,g)$ is said to be \emph{(weakly) $\sfrac14$-pinched} if there exists a positive function $\delta$ such that $\tfrac\delta4\leq\oc_g(\sigma)\leq\delta$ for all $\sigma$. This notion can be extended to any dimensions by requiring that the average of any two mutually orthogonal planes is $\sfrac14$-pinched.
As observed by Seaman~\cite{seaman2}, this pinching condition implies that the manifold has nonnegative isotropic curvature. It was later proved by Seaman~\cite{seaman}, and independently by Micallef and Wang~\cite{mw}, that if an even dimensional compact orientable manifold $(M,g)$ with $b_2(M)\neq0$ has nonnegative isotropic curvature and positive biorthogonal curvature at one point, then $(M,g)$ is K\"ahler, $b_2(M)=1$ and $M$ is simply-connected. Consequently, our metrics of positive biorthogonal curvature on $S^2\times S^2$ cannot satisfy the biorthogonal $\sfrac14$-pinching condition, since $b_2(S^2\times S^2)=2$. Moreover, it also follows that such metrics do not have nonnegative isotropic curvature.

\subsection{Modified Yamabe invariant}
As observed by Costa~\cite{ezio}, the minimum of the biorthogonal curvature at each point is a \emph{modified scalar curvature}, with  corresponding \emph{modified Yamabe invariant} denoted by $Y^\perp_1(M)=\sup_g Y_1^\perp(M,g)$, where the supremum is taken over all metrics $g$ on $M$. It is observed that if a metric $g\in [g_0]$ is conformal to the standard product metric on $S^2\times S^2$, then $Y_1^\perp(S^2\times S^2,g)\leq0$. In particular, no metric conformal to $g_0$ can have positive biorthogonal curvature. However, as a direct consequence of the Theorem in the Introduction, we have that $Y^\perp_1(S^2\times S^2)>0$, see \cite[Thm 3 (1)]{ezio}.

\section{Other $4$-manifolds with positive biorthogonal curvature}\label{sec:4mflds}

In light of the above construction, it is natural to inquire how restrictive the positive biorthogonal curvature condition is on $4$-manifolds. As noted before, $\oc>0$ automatically implies $\operatorname{scal}>0$, however it does not necessarily guarantee $\Ric>0$ (cf. Subsection~\ref{subsec:ricci}). On the one hand, this means that $\oc>0$ imposes rather restrictive topological conditions on $4$-manifolds, e.g., vanishing of all the Seiberg-Witten invariants. On the other hand, such topological restrictions are by far not as strong as the ones implied by $\sec>0$, or even $\Ric>0$. For instance, $\oc>0$ does not guarantee finiteness of the fundamental group, as illustrated by $S^1\times S^3$ with the standard product metric.

This suggests that more subtle Synge-type obstructions should also not detect $\oc>0$. In fact, $\R P^2 \times \R P^2$ admits metrics with $\oc>0$, as discussed in Subsection~\ref{subsec:rp2}. Another relevant example in this context is the nontrivial $S^2$-bundle over $S^2$, which is diffeomorphic to the connected sum $\C P^2\#\overline{\C P}^2$, where $\overline{\C P}^2$ denotes the manifold $\C P^2$ with the opposite orientation from the one induced by its complex structure. We conclude by showing that this manifold also has $\oc>0$, using arguments similar to the $S^2\times S^2$ case. It is important to observe that it is also currently unknown whether $\C P^2\#\overline{\C P}^2$ admits a metric with $\sec>0$.

\begin{proposition}\label{prop:cp2}
There manifold $\C P^2\#\overline{\C P}^2$ admits metrics with $\oc>0$.
\end{proposition}

\begin{proof}
Similarly to the $S^2\times S^2$ case, $\C P^2\#\overline{\C P}^2$ admits cohomogeneity one metrics with $\sec\geq0$ invariant under the action of $\SU(2)$. In order to describe this initial metric, notice that the normal bundle of the usual embedding $\C P^1\hookrightarrow\C P^2$ can be identified with the vector bundle $(S^3\times\R^2)/S^1$ over $\C P^1=S^3/S^1$. Take two copies of the disk bundles given as tubular neighborhoods of the zero section of this vector bundle. Each one of them is the complement of a metric ball on $\C P^2$, that is deleted to carry out the connected sum. It is then easy to see that
\begin{equation}
\C P^2\#\overline{\C P}^2=(S^3\times S^2)/S^1,
\end{equation}
by gluing along the boundary these two disk bundles. Here, the $S^1$-action on $S^3\times S^2$ is a product action, on $S^3$ via the Hopf action and on $S^2$ by rotation. The standard product metric on $S^3\times S^2$ then induces a submersion metric\footnote{We remark that this construction is very similar to the original gluing construction of nonnegatively curved metrics on the connected sum of two compact rank one symmetric spaces, which is due to Cheeger~\cite{cheeger} and was later greatly generalized by Grove and Ziller \cite{gzannals}. The only subtle difference is that $(\C P^2\#\overline{\C P}^2,g_0)$ has only one orbit that is a totally geodesic hypersurface (the boundary of the disk bundles glued together), while in the gluing construction the metric locally splits as a product near this hypersurface.}
 $g_0$ with nonnegative curvature on $\C P^2\#\overline{\C P}^2$. The cohomogeneity one action of $\SU(2)$ comes from the left-translation action of $\SU(2)=S^3$ on the first factor of $S^3\times S^2$, which induces an action on the quotient since it commutes with the above circle action.

Both singular orbits of this cohomogeneity one action on $(\C P^2\#\overline{\C P}^2,g_0)$ are $2$-spheres, that correspond to the zero section of the disk bundles that were glued together. The zero curvature planes are images via the submersion of mixed planes on $S^3\times S^2$ that are spanned by vectors orthogonal to the circle action field, cf. M\"uter~\cite[Satz 4.29]{mueter}. Thus, there is a circle's worth of zero curvature planes at every point, but any such planes tangent to a regular point \emph{must intersect}. At singular points, there are zero curvature planes orthogonal to each other, but all of them are not tangent to the singular orbit. This scenario is totally analogous to the Cheeger deformed metrics on $S^2\times S^2$, i.e., metrics obtained after the first step of our deformation (Proposition \ref{prop:geomcheeger}). More precisely, $\sec_{g_0}\geq0$, and $\oc_{g_0}>0$ on all regular points. Since the only points with zero biorthogonal curvature are along the singular orbits and all zero curvature planes are not tangent to these orbits, a first-order local conformal deformation using squared distance functions to the singular orbits, totally analogous to the one in Proposition~\ref{prop:final}, gives the desired metrics with $\oc>0$ on $\C P^2\#\overline{\C P}^2$ as a consequence of Proposition~\ref{prop:cpct}.
\end{proof}

\begin{remark}
As shown above, in order to construct metrics with $\oc>0$ on $\C P^2\#\overline{\C P}^2$, one can skip the first step in the construction for $S^2\times S^2$. This is an important observation, because differently from $S^2\times S^2$ with the standard product metric, the Cheeger deformation of $(\C P^2\#\overline{\C P}^2,g_0)$ with respect to the $\SU(2)$-action does not destroy \emph{any} zero curvature planes, see M\"uter~\cite[Satz 4.29]{mueter}.
\end{remark}

\begin{remark}
Since there is a circle's worth of zero curvature planes at every point on $(\C P^2\#\overline{\C P}^2,g_0)$, although the first-order local conformal deformation produces $\oc>0$, it cannot be used to produce metrics with $\sec^\theta>0$ for every $\theta>0$.
\end{remark}

\end{document}